\documentclass[a4paper,10pt,reqno]{amsart}

\title{Remarks on modulation spaces and homogeneity}

\usepackage{amssymb}
\usepackage{latexsym}
\usepackage{amsmath}
\usepackage{mathrsfs}
\usepackage[X2,T1]{fontenc}
\usepackage{amsthm}
\usepackage{upgreek}
\usepackage{color}
\usepackage{multicol}
\usepackage{pdfpages}
\usepackage{graphicx}
\usepackage{calc}                         
\newcommand{\scal}[2]{\langle #1,#2\rangle}

\newcommand{\nn}[1]{\mathbf N^{#1}}
\newcommand{\rr}[1]{\mathbf R^{#1}}

\newcommand{\cc}[1]{\mathbf C^{#1}}

\newcommand{\nm}[2]{\Vert #1\Vert _{#2}}

\newcommand{\sets}[2]{\{ \, #1\, ;\, #2\, \} }
\newcommand{\Sets}[2]{\left \{ \, #1\, ;\, #2\, \right \} }
\newcommand{\ep}{\varepsilon}

\newcommand{\cdo}{\, \cdot \, }

\newcommand{\eabs}[1]{\langle #1\rangle}

\newcommand{\vrum}{\vspace{0.1cm}}

\newcommand{\Arg}{\operatorname{Arg}}

\newcommand{\maclH}{\mathcal H}

\newcommand{\maclS}{\mathcal S}

\newcommand{\mascF}{\mathscr F}

\newcommand{\mascS}{\mathscr S}

\numberwithin{equation}{section}          
\newtheorem{thm}{Theorem}
\numberwithin{thm}{section}

\newcommand{\rubrik}{}
\newtheorem{prop}[thm]{Proposition}

\newtheorem{lemma}[thm]{Lemma}

\theoremstyle{definition}

\newtheorem{defn}[thm]{Definition}

\theoremstyle{remark}

\newtheorem{rem}[thm]{Remark}              

\title{Fourier characterizations of
Pilipovi{\'c} spaces}



\author{Joachim Toft}

\address{Department of Mathematics,
Linn{\ae}us University, V{\"a}xj{\"o}, Sweden}

\email{joachim.toft@lnu.se}

\author{Anupam Gumber}

\address{NuHAG, Faculty of Mathematics,
University of Vienna, Vienna, Austria}

\email{anupam.gumber@univie.ac.at}

\begin{document}

\begin{abstract}
Let $f$ be a function or distribution on $\rr d$.
We show that $f$ belongs to a certain Pilipovi{\'c} space,
if and only if $f$ and suitable partial fractional Fourier
transforms of $f$ satisfy certain types of estimates.
\end{abstract}

\maketitle

\section{Introduction}

\par

In the paper we find characterizations of
Pilipovi{\'c} spaces in terms of estimates
of suitable fractional Fourier transforms
of the involved functions.

\par

%
%
%
%
%
%
%
Pilipovi{\'c} spaces is a family of function
spaces 
which contains any standard Fourier invariant 
Gelfand-shilov space.
The Pilipovi{\'c} spaces $\maclH _s(\rr d)$ for
$s\ge 0$ and $\maclH _{0,s}(\rr d)$ for $s>0$,
contain all finite linear combinations of
Hermite functions and
are dense in the Schwartz space $\mascS (\rr d)$
(See \cite{Ho1} or Section \ref{sec1} for notations.)
One has
$$
\maclH _{s_1}(\rr d) \subseteq \maclH _{0,s_2}(\rr d)
\subseteq \maclH _{s_2}(\rr d) ,\qquad 0\le s_1<s_2.
$$
The Pilipovi{\'c} spaces increase with the parameter
$s$ above and are strongly related to the
Gelfand-Shilov spaces $\maclS _{s}(\rr d)$ and
$\Sigma _{s}(\rr d)$
of Roumieu and Beurling types, respectively,
which consist of all $f\in \mascS (\rr d)$ such that
\begin{equation}\label{Eq:GSCond}
\sup _{x\in \rr d}
\sup _{\alpha ,\beta \in \nn d}
\left (
\frac {|x^\alpha D^\beta f(x)|}
{h^{|\alpha +\beta |}(\alpha !\beta !)^s}
\right ) <\infty 
\end{equation}
holds true for some $h>0$ respectively
for every $h>0$.
In fact, we have
\begin{alignat}{4}
\maclH _{s_1}(\rr d) &= \maclS _{s_1}(\rr d), &
\quad
\maclH _{0,s_2}(\rr d) &= \Sigma _{s_2}(\rr d), &
s_2&>s_1\ge \frac 12
\label{Eq:IntroIdentPilGSSpaces1}
\intertext{but}
\maclH _{s_1}(\rr d) &\neq \maclS _{s_1}(\rr d)
= \{ 0 \} , &
\quad
\maclH _{0,s_2}(\rr d) &\neq  \Sigma _{s_2}(\rr d)
= \{ 0 \}, &
\quad
s_1&<s_2\le \frac 12 .
\label{Eq:IntroIdentPilGSSpaces2}
\end{alignat}
In particular, it follows that the functions
in Pilipovi{\'c} spaces obey strong
ultra-differentiability conditions and
strong exponential type decay bounds at infinity,
because similar facts hold true for Gelfand-Shilov
spaces.

\par

There are several characterizations of
Gelfand-Shilov spaces. One of the most important
is described in the following proposition,
which characterize such spaces in terms of estimates
of the involved elements and their Fourier transforms,
established by Eijndhoven in \cite{Eij}
and Chung, Chung and Kim in \cite{ChuChuKim}.

\par

\begin{prop}\label{Prop:IntroGSSpaceChar}
Let $s\ge 0$
and $f\in \mascS '(\rr d)$.
Then the following conditions are equivalent:
\begin{enumerate}
\item $f\in \maclS _s(\rr d)$
($f\in \Sigma _s(\rr d)$);

\vrum

\item $|f(x)|\lesssim e^{-r|x|^{\frac 1s}}$
and $|f(\xi )|\lesssim e^{-r|\xi |^{\frac 1s}}$
for some $r>0$ (for every $r>0$);
\end{enumerate}
\end{prop}

\par

Pilipovi{\'c} spaces can be defined in similar
ways when $s\ge 0$ is real. In fact, for such
$s$, Pilipovi{\'c}
spaces can be defined by replacing
the operators $x^\alpha$ and $D^\beta$
in \eqref{Eq:GSCond} with $H^{\frac 12}$,
where $H=|x|^2-\Delta _x$ is the harmonic
oscillator. That is, \eqref{Eq:GSCond}
should be replaced by
\begin{equation}\label{Eq:PilCond}
\sup _{x\in \rr d}
\sup _{N\ge 0}
\left (
\frac {|H^Nf(x)|}
{h^{N}N!^{2s}}
\right ) <\infty .
\end{equation}

\par

In the
following proposition we
characterize of Pilipovi{\'c} spaces
in terms of estimates
of the Hermite coefficients and
the short-time Fourier transform $V_\phi f$
of $f$.
Here we recall that any
$f\in \mascS '(\rr d)$
may in a unique way be expressed as a
Hermite series expansion
\begin{equation}\label{Eq:PilCond2}
f(x)= \sum _{\alpha \in \nn d}c(f;\alpha )
h_\alpha ,
\end{equation}
where $h_\alpha$ is the Hermite function of
order $\alpha \in \nn d$.

\par

\begin{prop}\label{Prop:IntroPilSpaceChar}
Let $s\ge 0$, $\phi (x) = \pi ^{-\frac d4}
e^{-\frac 12|x|^2}$,
\begin{equation}\label{Eq:WeightFunctionIntro1}
\vartheta _{r,s}(x,\xi )
=
\begin{cases}
e^{-r(|x|^{\frac 1s}+|\xi |^{\frac 1s})},
& s>\frac 12,\ r>0,
\\[1ex]
e^{-(\frac 14-r)(|x|^2+|\xi |^2)},
& s=\frac 12,\ 0<r<\frac 14,
\\[1ex]
e^{-\frac 14(|x|^2+|\xi |^2)
+r(\log (1+|x|+|\xi |))^{\frac 1{1-2s}}},
& s<\frac 12,\ r>0,
\end{cases}
\end{equation}
$f\in \mascS (\rr d)$ and let $c(f;\alpha )$
be given by \eqref{Eq:PilCond2}.
Then the following conditions are equivalent:
\begin{enumerate}
\item $f\in \maclH _s(\rr d)$
($f\in \maclH _{0,s}(\rr d)$);

\vrum

\item \eqref{Eq:PilCond} holds
for some $h>0$ (for every $h>0$);

\vrum

\item $|c(f;\alpha )|\lesssim e^{-r|\alpha |^{\frac 1{2s}}}$
for some $r>0$ (for every $r>0$);

\vrum

\item $|V_\phi f(x,\xi )|
\lesssim \vartheta _{r,s}(x,\xi )$
for some $r>0$ (for every $r>0$).
\end{enumerate}
\end{prop}

\par

In Proposition \ref{Prop:IntroPilSpaceChar},
the equivalence between (1) and (2) essentially
follows from the definitions, the
equivalence between (1) and (3)
was established in \cite{Pil1,Pil2}
for $s\ge \frac 12$ and in
\cite{Toft18} for general $s$.
The equivalence between (1) and (4)
was established in \cite{GZ} by Gr{\"o}chenig
and Zimmermann for $\maclH _s(\rr d)$ when $s\ge
\frac 12$ and for $\maclH _{0,s}(\rr d)$ when $s>
\frac 12$. For general $s$, the equivalence
between
(1) and (4) is obtained in \cite{Toft18}.  

\par

We observe that in view of 
\eqref{Eq:IntroIdentPilGSSpaces1},
it follows that Proposition 
\ref{Prop:IntroPilSpaceChar}
give some further characterizations for 
Gelfand-Shilov spaces.

\par

Despite Proposition 
\ref{Prop:IntroPilSpaceChar}
contains several characterizations for
Pilipovi{\'c} spaces, there are no
characterization of the form Proposition
\ref{Prop:IntroGSSpaceChar} in terms of
estimates of the involved functions and their
Fourier transforms. By modifying the
weight function $\vartheta _{r,s}(x,\xi )$ in
\eqref{Eq:WeightFunctionIntro1} into
\begin{equation}\label{Eq:WeightFunctionIntro2}
\omega _{r,s}(x)
=
\begin{cases}
e^{-r|x|^{\frac 1s}},
& s>\frac 12,\ r>0,
\\[1ex]
e^{-(\frac 12-r)|x|^2},
& s=\frac 12,\ 0<r<\frac 12,
\\[1ex]
e^{-\frac 12|x|^2
+r(\log (1+|x|))^{\frac 1{1-2s}}},
& s<\frac 12,\ r>0,
\end{cases}
\end{equation}
it follows by some manipulations that
\begin{align}
f\in \maclH _s(\rr d)\ (f\in \maclH _{0,s}(\rr d))
\quad &\Rightarrow \quad
|f(x)|\lesssim 
\omega _{r,s}(x),\ |\widehat f(\xi )|
\lesssim \omega _{r,s}(\xi ),
\label{Eq:IntroPilFourRel}
\intertext{
for some $r>0$ (for every $r>0$).
More generally, since $\maclH _s(\rr d)$
and $\maclH _{0,s}(\rr d)$ are invariant
under any (partial) fractional Fourier
transform \eqref{Eq:IntroPilFourRel}
generalizes into}
f\in \maclH _s(\rr d)\ (f\in \maclH _{0,s}(\rr d))
\quad &\Rightarrow \quad
|(\mascF _tf)(x)|\lesssim 
\omega _{r,s}(x),
\quad t\in \rr d,
\tag*{(\ref{Eq:IntroPilFourRel})$'$}
\intertext{for some $r>0$ (for every $r>0$).
(See Proposition \ref{Prop:PilSpGivesFTEst}
in Section \ref{sec2}.)
\indent
The originally searched result was to show that
equivalence occurs in \eqref{Eq:IntroPilFourRel},
which should give characterizations
of Pilipovi{\'c} spaces in terms of convenient
estimates of the involved functions and their
Fourier transforms. Unfortunately, those
techniques which we are aware of, seem to be
insufficient for such equivalence. On the other
hand, by some straight-forward estimates it turns
out that equivalence occurs in
\eqref{Eq:IntroPilFourRel}$'$ i.{\,}e.}
f\in \maclH _s(\rr d)\ (f\in \maclH _{0,s}(\rr d))
\quad &\Leftrightarrow \quad
|(\mascF _tf)(x)|\lesssim 
\omega _{r,s}(x),
\quad t\in \rr d,
\tag*{(\ref{Eq:IntroPilFourRel})$''$}
\intertext{for some $r>0$ (for every $r>0$),
giving a weaker form of searched
characterization. In fact, for suitable
lattices $\Lambda \subseteq \rr d$ we improve
\eqref{Eq:IntroPilFourRel}$''$ into}
f\in \maclH _s(\rr d)\ (f\in \maclH _{0,s}(\rr d))
\quad &\Leftrightarrow \quad
|(\mascF _tf)(x)|\lesssim 
\omega _{r,s}(x),
\quad t\in \Lambda ,
\tag*{(\ref{Eq:IntroPilFourRel})$'''$}
\end{align}
for some $r>0$ (for every $r>0$).
(See Theorems \ref{Thm:Mainthm1}
and \ref{Thm:Mainthm1A}.)

\par

In order to reach \eqref{Eq:IntroPilFourRel}$'''$,
we observe that the right implication follows from 
\eqref{Eq:IntroPilFourRel}$'$. The left implication
is managed in a different way, based on a
multi-dimensional version of Phragm{\'e}n-Lindel{\"o}f's
theorem, which seems not so easy to be found in the 
literature.
For this reason we have included an induction
proof of this result in Appendix \ref{AppA}, which
might be of some independent interest.

\par

The paper is organized as follows. In Section \ref{sec1}
we recall the definition and some basic
properties of Gelfand-Shilov spaces and
Pilipovi{\'c} spaces. In Section \ref{sec2} we
show the equivalences \eqref{Eq:IntroPilFourRel}$''$
and \eqref{Eq:IntroPilFourRel}$'''$. In Appendix
\ref{AppA} we deduce multi-dimensional estimates
of Phragm{\'e}n-Lindel{\"o}f type. In
Appendix \ref{AppB} we recall some important links 
between Bargmann transforms and short-time Fourier
transforms with Gaussian windows, and discuss
compositions of such transforms with fractional
Fourier transforms.



%
%

\par

\section*{Acknowledgement}
The first author was supported by
Vetenskapsr{\aa}det (Swedish 
Science Council) within the project 2019-04890.
The second author is supported by the Austrian
Science Fund (FWF) project TAI6.

\par

\section{Preliminaries}\label{sec1}

\par

\subsection{Gelfand-Shilov spaces}\label{subsec1.2}
Let $0<s \in \mathbf R$ be fixed. Then the (Fourier invariant)
Gelfand-Shilov space $\maclS _s(\rr d)$ ($\Sigma _s(\rr d)$) of
Roumieu type (Beurling type) consists of all $f\in C^\infty (\rr d)$
such that
\begin{equation}\label{gfseminorm}
\nm f{\maclS _{s,h}^\sigma}\equiv \sup \frac {|x^\alpha \partial ^\beta
f(x)|}{h^{|\alpha  + \beta |}(\alpha !\, \beta !)^s}
\end{equation}
is finite for some $h>0$ (for every $h>0$). Here the supremum should be taken
over all $\alpha ,\beta \in \mathbf N^d$ and $x\in \rr d$. The semi-norms
$\nm \cdo {\maclS _{s,h}^\sigma}$ induce an inductive limit topology for the
space $\maclS _s(\rr d)$ and projective limit topology for $\Sigma _s(\rr d)$, and
the latter space becomes a Fr{\'e}chet space under this topology.

\par

The space $\maclS _s(\rr d)\neq \{ 0\}$ ($\Sigma _s(\rr d)\neq \{0\}$), if and only if
$s\ge \frac 12$ ($s> \frac 12$).

\par

The \emph{Gelfand-Shilov distribution spaces} $\maclS _s'(\rr d)$
and $\Sigma _s'(\rr d)$ are the dual spaces of $\maclS _s(\rr d)$
and $\Sigma _s(\rr d)$, respectively.

\par

We have
\begin{equation}\label{GSembeddings}
\begin{aligned}
\maclS _{1/2} (\rr d) &\hookrightarrow \Sigma _s  (\rr d) \hookrightarrow
\maclS _s (\rr d)
\hookrightarrow  \Sigma _t(\rr d)
\\[1ex]
&\hookrightarrow
\mascS (\rr d)
\hookrightarrow \mascS '(\rr d) 
\hookrightarrow \Sigma _t' (\rr d)
\\[1ex]
&\hookrightarrow  \maclS _s'(\rr d)
\hookrightarrow  \Sigma _s'(\rr d) \hookrightarrow \maclS _{1/2} '(\rr d),
\quad \frac 12<s<t.
\end{aligned}
\end{equation}
Here and
in what follows we use the notation $A\hookrightarrow B$ when the topological
spaces $A$ and $B$ satisfy $A\subseteq B$ with continuous embeddings.

\par

A convenient family of functions concerns the Hermite functions
$$
h_\alpha (x) = \pi ^{-\frac d4}(-1)^{|\alpha |}
(2^{|\alpha |}\alpha !)^{-\frac 12}e^{\frac {|x|^2}2}
(\partial ^\alpha e^{-|x|^2}),\quad \alpha \in \nn d.
$$
The set of Hermite functions on $\rr d$ is an orthonormal basis for
$L^2(\rr d)$. It is also a basis for the Schwartz space and its distribution space,
and for any $\Sigma _s$ when $s>\frac 12$,
$\maclS _s$ when $s\ge \frac 12$ and their distribution
spaces. They are also eigenfunctions to the Harmonic
oscillator $H=H_d\equiv |x|^2-\Delta$ and to the Fourier transform
$\mathscr F$, given by
$$
(\mathscr Ff)(\xi )= \widehat f(\xi ) \equiv (2\pi )^{-\frac d2}\int _{\rr
{d}} f(x)e^{-i\scal  x\xi }\, dx, \quad \xi \in \rr d,
$$
when $f\in L^1(\rr d)$. Here $\scal \cdo \cdo$ denotes the usual
scalar product on $\rr d$. In fact, we have
$$
H_dh_\alpha = (2|\alpha |+d)h_\alpha .
$$

\par

The Fourier transform $\mathscr F$ extends
uniquely to homeomorphisms on $\mathscr S'(\rr d)$,
$\maclS _s'(\rr d)$ and on $\Sigma _s'(\rr d)$. Furthermore,
$\mascF$ restricts to
homeomorphisms on $\mathscr S(\rr d)$,
$\maclS _s(\rr d)$ and on $\Sigma _s (\rr d)$,
and to a unitary operator on $L^2(\rr d)$. Similar facts hold true
when the Fourier transform is replaced by a partial
Fourier transform.

\par

Gelfand-Shilov spaces and their distribution spaces can also
be characterized by estimates of short-time Fourier
transform, (see e.{\,}g. \cite{GZ,Teof,Toft18}).
More precisely, let $\phi \in \mascS  (\rr d)$ be
fixed.
Then the \emph{short-time
Fourier transform} $V_\phi f$ of $f\in \mascS '
(\rr d)$ with respect to the \emph{window function} $\phi$ is
the Gelfand-Shilov distribution on $\rr {2d}$, defined by
$$
V_\phi f(x,\xi )  =
\mascF (f \, \overline {\phi (\cdo -x)})(\xi ), \quad x,\xi \in \rr d.
$$
If $f ,\phi \in \mascS (\rr d)$, then it follows that
$$
V_\phi f(x,\xi ) = (2\pi )^{-\frac d2}\int _{\rr d} f(y)\overline {\phi
(y-x)}e^{-i\scal y\xi}\, dy, \quad x,\xi \in \rr d.
$$

\par

By \cite[Theorem 2.3]{To11} it follows that the definition of the map
$(f,\phi)\mapsto V_{\phi} f$ from $\mascS (\rr d) \times \mascS (\rr d)$
to $\mascS(\rr {2d})$ is uniquely extendable to a continuous map from
$\maclS _s'(\rr d)\times \maclS_s'(\rr d)$
to $\maclS_s'(\rr {2d})$, and restricts to a continuous map
from $\maclS _s (\rr d)\times \maclS _s (\rr d)$
to $\maclS _s(\rr {2d})$.
The same conclusion holds with $\Sigma _s$ in place of
$\maclS_s$, at each place.

\par

In the following propositions we give characterizations of Gelfand-Shilov
spaces and their distribution spaces in terms of estimates of the short-time Fourier transform.
We omit the proof since the first part follows from \cite[Theorem 2.7]{GZ})
and the second part from \cite[Proposition 2.2]{Toft18}.
See also \cite{CPRT10} for related results.

\par

\begin{prop}\label{stftGelfand2}
Let $s\ge \frac 12$ ($s>\frac 12$), $\phi \in \maclS _s(\rr d)\setminus 0$
($\phi \in \Sigma _s(\rr d)\setminus 0$) and let $f$ be a
Gelfand-Shilov distribution on $\rr d$. Then the following is true:
\begin{enumerate}
\item $f\in \maclS _s (\rr d)$ ($f\in \Sigma_s(\rr d)$), if and only if
\begin{equation}\label{stftexpest2}
|V_\phi f(x,\xi )| \lesssim  e^{-r (|x|^{\frac 1s}+|\xi |^{\frac 1s})}, \quad x,\xi \in \rr d,
\end{equation}
for some $r > 0$ (for every $r>0$).
\item $f\in \maclS _s'(\rr d)$ ($f\in \Sigma _s'(\rr d)$), if and only if
\begin{equation}\label{stftexpest2Dist}
|V_\phi f(x,\xi )| \lesssim  e^{r(|x|^{\frac 1s}+|\xi |^{\frac 1s})}, \quad
x,\xi \in \rr d,
\end{equation}
for every $r > 0$ (for some $r > 0$).
\end{enumerate}
\end{prop}

\par

\subsection{Spaces of Hermite series expansions}
\label{subsec1.3}

\par

Next we recall the definitions of topological vector spaces of
Hermite series expansions, given in \cite{Toft18}. As in \cite{Toft18},
it is convenient to use suitable extensions of
$\mathbf R_+$ when indexing our spaces.

\par

\begin{defn}
The sets $\mathbf R_\flat$ and $\overline {\mathbf R_\flat}$ are given by
$$
{\textstyle{\mathbf R_\flat = \mathbf R_+ \underset{\sigma >0}{\textstyle{\bigcup}}
\{ \flat _\sigma \} }}
\quad \text{and}\quad
{\textstyle{\overline {\mathbf R_\flat} = \mathbf R_\flat \bigcup \{ 0 \} }}.
$$

\par

Moreover, beside the usual ordering in $\mathbf R$, the elements $\flat _\sigma$
in $\mathbf R_\flat$ and $\overline {\mathbf R_\flat}$ are ordered by
the relations $x_1<\flat _{\sigma _1}<\flat _{\sigma _2}<x_2$, when
$\sigma _1$, $\sigma _2$, $x_1$ and $x_2$ are positive real numbers such that
$x_1<\frac 12$ and $x_2\ge \frac 12$.
\end{defn}

\par

\begin{defn}\label{DefSeqSpaces}
Let $p\in [1,\infty ]$, $s\in {\mathbf R_\flat}$, $r\in \mathbf R$, $\vartheta$
be a weight on $\nn d$, and let
$$
\vartheta _{r,s}(\alpha )\equiv
\begin{cases}
e^{r|\alpha |^{\frac 1{2s}}}, & \text{when}\quad s\in \mathbf R_+,
\\[1ex]
r^{|\alpha |}(\alpha !)^{\frac 1{2\sigma}}, & \text{when}\quad s = \flat _\sigma ,
\quad \qquad \alpha \in \nn d.
\end{cases}
$$
Then,
\begin{enumerate}
\item $\ell _0' (\nn d)$ is the set of all sequences $\{c(\alpha ) \} _{\alpha \in \nn d}
\subseteq \mathbf C$ on $\nn d$;

\vrum

\item $\ell _{0,0}(\nn d)\equiv \{ 0\}$, and $\ell _0(\nn d)$ is the set of all sequences
$\{c(\alpha ) \} _{\alpha \in \nn d}\subseteq \mathbf C$ such that $c(\alpha ) \neq 0$
for at most finite numbers of $\alpha$;

\vrum

\item $\ell ^p_{[\vartheta ]}(\nn d)$ is the Banach space which consists of
all sequences $\{ c(\alpha ) \} _{\alpha \in \nn d} \subseteq \mathbf C$
such that
$$
\nm {\{ c(\alpha ) \} _{\alpha \in \nn d} }{\ell ^p_{[\vartheta ]}}\equiv
\nm {\{ c(\alpha ) \vartheta (\alpha )\} _{\alpha \in \nn d} }{\ell ^p} < \infty ;
$$

\vrum

\item $\ell _{0,s}(\nn d)\equiv \underset {r>0}\bigcap \ell ^p_{[\vartheta _{r,s}]}(\nn d)$
and $\ell _s(\nn d)\equiv \underset {r>0}\bigcup \ell ^p_{[\vartheta _{r,s}]}(\nn d)$, with
projective respective inductive limit topologies of $\ell ^p_{[\vartheta _{r,s}]}(\nn d)$
with respect to $r>0$;

\vrum

\item $\ell _{0,s}'(\nn d)\equiv \underset {r>0}\bigcup
\ell ^p_{[1/\vartheta _{r,s}]}(\nn d)$ and $\ell _s'(\nn d)\equiv \underset {r>0}
\bigcap \ell ^p_{[1/\vartheta _{r,s}]}(\nn d)$, with
inductive respective projective limit topologies of $\ell ^p_{[1/\vartheta _{r,s}]}(\nn d)$
with respect to $r>0$.
\end{enumerate}
\end{defn}

\par

Let $p\in [1,\infty ]$, and let $\Omega _N$ be the set of all
$\alpha \in \nn d$ such that $|\alpha |\le N$. Then the
topology of $\ell _0(\nn d)$ is defined by the inductive
limit topology of the sets
$$
\Sets {\{ c(\alpha ) \} _{\alpha \in \nn d} \in \ell _0'(\nn d)}{c(\alpha ) =0\
\text{when}\ \alpha \notin \Omega _N}
$$
with respect to $N\ge 0$, and whose topology
is given through the semi-norms
\begin{equation}\label{SemiNormsEllSpaces}
\{ c(\alpha ) \} _{\alpha \in \nn d}\mapsto \nm {\{ c(\alpha ) \}
_{|\alpha |\le N} }{\ell ^p(\Omega _N)}.
\end{equation}
It is clear that these topologies are independent of $p$.
Furthermore, the topology of
$\ell _0' (\nn d)$ is defined by the semi-norms \eqref{SemiNormsEllSpaces}.
It follows that $\ell _0'(\nn d)$ is a Fr{\'e}chet space, and that the topology
is independent of $p$.

\par

Next we introduce spaces of formal Hermite series expansions
\begin{alignat}{2}
f&=\sum _{\alpha \in \nn d}c(f,\alpha ) h_\alpha ,&\quad \{ c(f,\alpha ) \}
_{\alpha \in \nn d} &\in \ell _0' (\nn d),\label{Eq:Hermiteseries}
\end{alignat}
which correspond to
\begin{equation}\label{ellSpaces}
\ell _{0,s}(\nn d),\quad \ell _s(\nn d),
\quad \ell _s'(\nn d)\quad \text{and}\quad \ell _{0,s}'(\nn d).
\end{equation}
We consider the mappings
\begin{equation}
\label{T12Map}
T_{\maclH} : \,
\{ c(\alpha )\} _{\alpha \in \nn d} \mapsto \sum _{\alpha \in \nn d}
c(\alpha )h_\alpha
\end{equation}
between sequences and formal Hermite series expansions.

\par

\begin{defn}\label{DefclHclASpaces}
If $s\in \overline{\mathbf R_\flat}$, then
\begin{alignat}{2}
\maclH _{0,s}(\rr d),\quad \maclH _s(\rr d),
\quad \maclH _s'(\rr d)\quad \text{and}\quad \maclH _{0,s}'(\rr d),\label{clHSpaces}
\end{alignat}
are the images of $T_{\maclH}$ respectively in \eqref{T12Map}
of corresponding spaces in \eqref{ellSpaces}.
The topologies of the spaces in \eqref{clHSpaces} 
are inherited from the corresponding spaces in \eqref{ellSpaces}.
\end{defn}

\par


\par

We recall that $f\in \mascS (\rr d)$ if and only if it can be written as
\eqref{Eq:Hermiteseries} such that
$$
|c(f,\alpha ) |\lesssim \eabs \alpha ^{-N},
$$
for every $N\ge 0$ (cf. e.{\,}g. \cite{RS}).
In particular it follows from the definitions that the
inclusions
\begin{multline}\label{inclHermExpSpaces}
\maclH _0(\rr d)\hookrightarrow \maclH _{0,s}(\rr d)\hookrightarrow
\maclH _{s}(\rr d) \hookrightarrow \maclH _{0,t}(\rr d)
\\[1ex]
\hookrightarrow \mascS (\rr d) \hookrightarrow \mascS '(\rr d)
\hookrightarrow
\maclH _{0,t}'(\rr d) \hookrightarrow \maclH _{s}'(\rr d)
\\[1ex]
\hookrightarrow \maclH _{0,s}'(\rr d)\hookrightarrow \maclH _0'(\rr d),
\quad \text{when}\ s,t\in \mathbf R_\flat ,\ s<t,
\end{multline}
are dense.

\par

\begin{rem}\label{Rem:LinkEllHs}
By the definition it follows that $T_{\maclH}$ in \eqref{T12Map} is a homeomorphism
between any of the spaces in \eqref{ellSpaces} and corresponding space
in \eqref{clHSpaces}.
\end{rem}

\par

The next results give some characterizations of $\maclH _s(\rr d)$ and
$\maclH _{0,s}(\rr d)$ when $s$ is a non-negative real number.

\par

\begin{prop}\label{Prop:PilSpacChar1}
Let $0\le s\in \mathbf R$ and let $f\in \maclH _0'(\rr d)$.
Then $f\in \maclH _s(\rr d)$ ($f\in \maclH _{0,s}(\rr d)$), if and only if
$f\in C^\infty (\rr d)$ and satisfies
\begin{equation}\label{GFHarmCond}
\nm{H_d^Nf}{L^\infty}\lesssim h^NN!^{2s},
\end{equation}
for some $h>0$ (every $h>0$). Moreover, it holds
\begin{align*}
\maclH _s(\rr d) &= \maclS _s(\rr d) \neq \{ 0\},
\quad
\maclH _{0,s}(\rr d) = \Sigma _s(\rr d) \neq \{ 0\} 
\quad \text{when}
\quad s\in (\textstyle{\frac 12},\infty),
\\[1ex]
\maclH _s(\rr d) &= \maclS _s(\rr d) \neq \{ 0\},
\quad
\maclH _{0,s}(\rr d) \neq \Sigma _s(\rr d) = \{ 0\} 
\quad \text{when}
\quad s=\textstyle{\frac 12},
\\[1ex]
\maclH _s(\rr d) &\neq \maclS _s(\rr d) = \{ 0\},
\quad
\maclH _{0,s}(\rr d) \neq \Sigma _s(\rr d) = \{ 0\} 
\quad \text{when}
\quad s\in (0,\textstyle{\frac 12}),
\\[1ex]
\maclH _s(\rr d) &\neq \maclS _s(\rr d) = \{ 0\},
\quad
\maclH _{0,s}(\rr d) = \Sigma _s(\rr d) = \{ 0\} 
\quad \text{when}
\quad s=0.
\end{align*}
\end{prop}

\par

We refer to \cite{Toft18} for the proof of Proposition \ref{Prop:PilSpacChar1}.

\par

Due to the pioneering investigations related
to Proposition \ref{Prop:PilSpacChar1} by Pilipovi{\'c} in
\cite{Pil1,Pil2}, we call the spaces
$\maclH _s(\rr d)$ and $\maclH _{0,s}(\rr d)$
\emph{Pilipovi{\'c} spaces of Roumieu and
Beurling types}, respectively.
In fact, in the restricted case $s\ge \frac 12$,
Proposition \ref{Prop:PilSpacChar1} was proved
already in \cite{Pil1,Pil2}.

\par

\section{Characterizations of Pilipovi{\'c}
spaces by estimates of the Fourier transform}
\label{sec2}

\par






In this section we deduce characterizations of
Pilipovi{\'c} spaces in terms of estimates
of the involved functions and some of
their fractional Fourier transforms. These
main issues are given in Theorems
\ref{Thm:Mainthm1} and \ref{Thm:Mainthm1A}.

\par

The following characterization of Pilipovi{\'c}
spaces of
lower orders is a straight-forward
consequence of Theorems 5.3, 6.2, (2.25) in 
\cite{Toft18}, and
Theorem 3.2 in \cite{To11}. The details are left
for the reader. Here
\begin{equation}\label{Eq:omegaDef}
\omega _{d,r,s} (z)
=
\begin{cases}
e^{-r\log \eabs z^{\frac 1{1-2s}}},
& s\in \mathbf R,\ 0\le s <\frac 12,
\\[1ex]
e^{-r|z|^{\frac {2\sigma}{\sigma +1}}},
& s=\flat _\sigma ,
\\[1ex]
e^{-r|z|^2}, & s=\frac 12,\quad z\in \cc d .
\end{cases}
\end{equation}
We observe that for every integer $d\ge 1$, $r\ge 0$ and
$s\in \overline{\mathbf R_\flat}$ with $s\le \frac 12$, it follows that
$\omega _{d,r,s}$ in \eqref{Eq:omegaDef} satisfies
$$
C^{-1}\omega _{d,r,s} (x) \le \omega _{d,r,s} (x+y)\le C\omega _{d,r,s} (x)
\quad \text{when}\quad
Rc\le |x|\le c/|y|,
$$
for some $R\ge 2$ and positive constants $c$ and $C$. This is needed for applying
Theorem 3.2 in \cite{To11}.

\par

\begin{prop}\label{Prop:PilSTFTChar}
Let $p\in (0,\infty ]$, $s\in \overline{\mathbf R_\flat}$ be such that $s\le \frac 12$,
$\phi = \pi ^{-\frac d4}e^{-\frac 12|x|^2}$, $f\in \maclH _0'(\rr d)$,
and let $\omega _{d,r,s}$ be given by \eqref{Eq:omegaDef}.
Then the following is true:
\begin{enumerate}
\item if in addition $s<\frac 12$, then $f\in \maclH _s(\rr d)$, if
and only if
\begin{equation}\label{Eq:PilSTFTChar}
\nm {V_\phi f \cdot e^{\frac 14|\cdo |^2}\omega _{2d,r,s}}{L^p}
\end{equation}
is finite for some $r>0$;

\vrum

\item if in addition $s>0$, then $f\in \maclH _{0,s}(\rr d)$, if
and only if \eqref{Eq:PilSTFTChar} is finite for every $r>0$.
\end{enumerate}
\end{prop}

\par

We have the following, where we show that elements in
$\maclH _0'(\rr d)$ belong to $\maclH _s(\rr d)$ or
$\maclH _{0,s}(\rr d)$, if and only if they satisfy estimates of the
form
\begin{align}
\sup _{t\in \Lambda _{t_0,u}}
\left (
\nm {(\mascF _tf)\cdot e^{\frac 12|\cdo |^2}\omega _{d,r,s}}{L^\infty}
\right )
&<
\infty ,
\notag
\\
\Lambda _{t_0,u} = \sets {t_0+(k_1u_{1},\dots ,k_du_{d})}{0 &\le k_ju_{j}
<2,
\ k_j\in \mathbf Z},
\label{Eq:PilFTChar1}
\intertext{or}
\sup _{t\in \rr d}
\left (
\nm {(\mascF _tf)\cdot e^{\frac 12|\cdo |^2}\omega _{d,r,s}}{L^\infty}
\right )
&<
\infty .
\label{Eq:PilFTChar2}
\end{align}








\begin{thm}\label{Thm:Mainthm1}
Let
$s\in \overline{\mathbf R_\flat}$ be such that $s< \frac 12$,
$t_0\in \rr d$, $u\in (0,1)^d$,
$f\in \maclH _0'(\rr d)$,
and let $\omega _{d,r,s}$ be given by \eqref{Eq:omegaDef}.
Then the following conditions are equivalent:
\begin{enumerate}
\item $f\in \maclH _s(\rr d)$;

\vrum

\item \eqref{Eq:PilFTChar1} holds for some $r>0$;

\vrum

\item \eqref{Eq:PilFTChar2} holds for some $r>0$.
\end{enumerate}
\end{thm}
\par

\begin{thm}\label{Thm:Mainthm1A}
Let
$s\in \mathbf R_\flat$ be such that
$s\le \frac 12$,
$t_0\in \rr d$, $u\in (0,1]^d$ for $s=\frac 12$ and
$u\in (0,1)^d$ for $s<\frac 12$,
$f\in \maclH _0'(\rr d)$,
and let $\omega _{d,r,s}$ be given by \eqref{Eq:omegaDef}.
Then the following conditions are equivalent:
\begin{enumerate}
\item $f\in \maclH _{0,s}(\rr d)$;

\vrum

\item \eqref{Eq:PilFTChar1} holds for every $r>0$;

\vrum

\item \eqref{Eq:PilFTChar2} holds for every $r>0$.
\end{enumerate}
\end{thm}

\par






\par



\par

We need some preparations for the proofs.
First we have the following proposition.

\par

\begin{prop}\label{Prop:PilSpGivesFTEst}
Let $s\in \overline{\mathbf R_\flat}$ be such that $s< \frac 12$ ($0<s\le \frac 12$),
$f\in \maclH _s(\rr d)$ ($f\in \maclH _{0,s}(\rr d)$), and let
$\omega _{d,r,s}$ be given by \eqref{Eq:omegaDef}. Then
\begin{equation}\label{Eq:PilSpGivesFTEst}
\sup _{t\in \rr d}
\nm {\mascF _tf \cdot e^{\frac 12|\cdo |^2}\omega _{d,r,s}}{L^\infty}<\infty
\end{equation}
holds for some $r>0$ (for every $r>0$).
\end{prop}

\par

\begin{proof}
We only prove the result for
$s>0$.
The case $s=0$ follows
by similar arguments and is left for the reader.

\par

Let $\phi (x)=\pi ^{-\frac d4}e^{-\frac 12|x|^2}$. By Fourier inversion formula we get
\begin{equation*}
\left | \int _{\rr d} V_\phi f(x,\xi )e^{\frac i2\scal x\xi}\, d\xi   \right |
\asymp
|f({\textstyle{\frac 12}}x)\phi ({\textstyle{\frac 12}x}-x)|
\asymp
|f({\textstyle{\frac 12}}x)|e^{-\frac 18|x|^2}.
\end{equation*}
This gives
\begin{multline*}
|f({\textstyle{\frac 12}}x)|e^{-\frac 18|x|^2}
\asymp
\left | \int _{\rr d} V_\phi f(x,\xi )e^{\frac i2\scal x\xi}\, d\xi   \right |
\le
\int _{\rr d} |V_\phi f(x,\xi )|\, d\xi 
\\[1ex]
\lesssim
\int _{\rr d} e^{-\frac 14(|x|^2+|\xi |^2)}\omega _{2d,r,s}(x,\xi )\, d\xi
\\[1ex]
\lesssim
\int _{\rr d} e^{-\frac 14(|x|^2+|\xi |^2)}\omega _{d,c_1r,s}(x)\omega _{d,c_1r,s}(\xi )
\, d\xi
\asymp
e^{-\frac 14|x|^2}\omega _{d,c_1r,s}(x)
\end{multline*}
for some $r>0$ (every $r>0$).
Hence,
$$
|f({\textstyle{\frac 12}}x)| \lesssim e^{-\frac 14|x|^2}\omega _{d,c_1r,s}(x)
e^{\frac 18|x|^2}
=
e^{-\frac 18|x|^2}\omega _{d,c_1r,s}(x)
$$
for some $r>0$ (every $r>0$). Here the positive constant $c_1$ can
be chosen independently on $r$.
This gives
$$
|f(x)| \lesssim e^{-\frac 12|x|^2}\omega _{d,c_2r,s}(x)
$$
for some $r>0$ (every $r>0$), and the assertion holds for $t=0$.
Since $\{ \mascF _t \} _{t\in \rr d}$
is an equicontinuous set of homeomorphisms on $\maclH_s(\rr d)$
and $\maclH_{0,s}(\rr d)$ in view of \cite[Proposition 7.1]{Toft18} and
its proof, it follows by replacing $f$ by $\mascF _tf$ in the previous
estimates that \eqref{Eq:PilSpGivesFTEst} holds for some $r>0$
(every $r>0$).
\end{proof}

\par

Proposition \ref{Prop:PilSpGivesFTEst} shows one of the
directions in Theorems \ref{Thm:Mainthm1} and \ref{Thm:Mainthm1A}.
The converse needs more steps and we begin with the following. Here
we involve the Bargmann transform because later investigations 
are based on analyticity arguments. In what follows we let
\begin{equation}\label{Eq:AtMatrixDef}
\begin{aligned}
A_t(x,\xi ) &= (\cos (t\textstyle{\frac \pi 2})x +
\sin (t\textstyle{\frac \pi 2})\xi ,
-\sin (t\textstyle{\frac \pi 2})x +\cos (t\textstyle{\frac \pi 2})\xi ),
\quad
x,t, \xi \in \mathbf R 
\\[1ex]
A_{d,t}(x,\xi ) &= (A_{t_1}(x_1,\xi _1),\dots ,A_{t_d}(x_d,\xi _d)),
\quad
x,t,\xi \in \rr d.
\end{aligned}
\end{equation}

\par

\begin{prop}\label{Prop:FTEstGivesWeakPilSpProp}
Let $s\in \overline{\mathbf R_\flat}$ be such that $s< \frac 12$
($0<s\le \frac 12$),
$t\in \rr d$, $\phi (x)=\pi ^{-\frac d4}e^{-\frac 12|x|^2}$
and let $\omega _{d,r,s}$ and $A_t(x,\xi )$ be given by
\eqref{Eq:omegaDef} and \eqref{Eq:AtMatrixDef}.
Then there is a constant $c>0$, and for every $r>0$, a constant $C=C_r>0$
such that
\begin{alignat}{2}
\sup _{x,\xi \in \rr d}
|V_\phi f(A_{d,t}(x,\xi ))e^{\frac 14|x|^2}\omega _{d,r,s}(x)|
&\le
C_r\nm {\mascF _tf \cdot e^{\frac 12|\cdo |^2}\omega _{d,c r,s}}{L^\infty}
\label{Eq:FTEstGivesWeakPilSpProp1}
\end{alignat}
and
\begin{multline}
\label{Eq:FTEstGivesWeakPilSpProp2}
\sup _{z \in \cc d}
|\mathfrak V_df(e^{-it_1\frac \pi 2}z_1,\dots
,e^{-it_d\frac \pi 2}z_d)\cdot \omega _{d,r,s}
(\operatorname{Re}(z))e^{-\frac 12|\operatorname{Im}(z)|^2}|
\\[1ex]
\le
C_r\nm {\mascF _tf \cdot e^{\frac 12|\cdo |^2}\omega _{d,cr,s}}{L^\infty}
\end{multline}
for every $f\in \maclH _0'(\rr d)$.
\end{prop}

\par

\begin{proof}
The estimates \eqref{Eq:FTEstGivesWeakPilSpProp1}
and \eqref{Eq:FTEstGivesWeakPilSpProp2} are equivalent
in view of Appendix \ref{AppB} (see also \cite{Toft18}). 
Since
$$
(\mathfrak V_d(\mascF _tf))(z)=
(\mathfrak V_df)(e^{-it_1\frac \pi 2}z_1,\dots
,e^{-it_d\frac \pi 2}z_d),
\quad z=(z_1,\dots ,z_d),\ t=(t_1,\dots ,t_d),
$$
the proposition follows if we prove 
\eqref{Eq:FTEstGivesWeakPilSpProp1} in the case $t=0$.

\par

By Hausdorff-Young's inequality we get for some
constants $c_1,c_2>0$,
\begin{align*}
\sup _{x,\xi \in \rr d}
&| V_\phi f(x,\xi ) e^{\frac 14|x|^2}\omega _{d,r,s}(x)|
\\[1ex]
&=
\sup _{x,\xi \in \rr d}
|\mascF (f\cdot \overline {\phi (\cdo -x)})(\xi )
e^{\frac 14|x|^2}\omega _{d,r,s}(x)|
\\[1ex]
&\le
\sup _{x \in \rr d}
\left (
\int |f(y)\phi (y-x) e^{\frac 14|x|^2}\omega _{d,r,s}(x)|\, dy
\right )
\\[1ex]
&=
\sup _{x \in \rr d}
\left (
\int |(f(y)e^{\frac 12|y|^2})e^{-\frac 14|x-2y|^2}\omega _{d,r,s}(x) |\, dy
\right )
\\[1ex]
&\lesssim
\sup _{x \in \rr d}
\left (
\int |(f(y)e^{\frac 12|y|^2}\omega _{d,c_1r,s}(y))(e^{-\frac 14|x-2y|^2}/
\omega _{d,c_2r,s}(x-2y))|\, dy
\right )
\\[1ex]
&\le
\sup _{x \in \rr d}
\left (
\int e^{-\frac 14|x-2y|^2}/
\omega _{d,c_2r,s}(x-2y)\, dy
\right )
\sup _{y\in \rr d} |f(y)e^{\frac 12|y|^2}\omega _{d,c_1r,s}(y)|
\\[1ex]
&\asymp
\sup _{y\in \rr d} |f(y)e^{\frac 12|y|^2}\omega _{d,c_1r,s}(y)|.
\end{align*}
That is,
$$
\sup _{x,\xi \in \rr d}
|V_\phi f(x,\xi ) e^{\frac 14|x|^2}\omega _{d,r,s}(x)|
\lesssim
\sup _{x\in \mathbf R}|f(x)e^{\frac 12|x|^2}\omega _{d,cr,s}(x)|
$$
for some $c>0$ which is independent of $r>0$. This gives
\eqref{Eq:FTEstGivesWeakPilSpProp1}.
\end{proof}

\par

We also have the following.

\par

\begin{prop}\label{Prop:EssEst}
Let $s\in \overline{\mathbf R_\flat}$,
$\omega _{d,r,s}$ be given by \eqref{Eq:omegaDef},
and let
$$
\Omega =\Gamma _1\times \cdots \times \Gamma _d
\quad \text{and}\quad
\Omega _0
=\partial \Gamma _1\times \cdots
\times \partial \Gamma _d,
$$
where
$$
\Gamma _j = \sets {z\in \mathbf C}
{\alpha _j\le \Arg (z)\le \beta _j},
\quad j=1,\dots ,d,
$$
and $\alpha _j,\beta _j\in \mathbf R$ satisfy
$0<\beta _j-\alpha _j<\frac \pi 2$. If
$F\in A(\Omega )$ satisfies
\begin{equation}\label{Eq:EssEst}
|F(z)|\lesssim e^{r_0|z|^2},\ z\in \Omega,
\quad \text{and}\quad |F(z)|\lesssim
\omega _{d,r,s}(z)^{-1},\ z\in \Omega _0,
\end{equation}
for some constants $r_0,r>0$, then for
some constant $c$ which is independent
of $r$ and $r_0$, it holds
$$
|F(z)|\lesssim \omega _{d,c\cdot r,s}(z)^{-1},\
z\in \Omega .
$$
\end{prop}

\par

For the proof we need the following lemma, which is a special case
of Proposition \ref{Prop:PhragmLindelMultDim} in Appendix
\ref{AppA}. The proof is therefore omitted.

\par

\begin{lemma}\label{Lemma:EssEst}
Let $\Omega$ and $\Omega _0$ be the same
as in Proposition \ref{Prop:EssEst} and
suppose that
$$
|F(z)|\lesssim e^{r_0|z|^2},\ z\in \Omega,
\quad \text{and}\quad |F(z)|\le M,\
z\in \Omega _0,
$$
for some constants $r_0,M>0$. Then
$$
|F(z)|\le M,\ z\in \Omega .
$$
\end{lemma}

\par

\begin{proof}[Proof of Proposition \ref{Prop:EssEst}]
By choosing
$\theta _1,\dots ,\theta _d\in \mathbf R$
in suitable ways and considering
$$
G(z_1,\dots ,z_d) = F(e^{i\theta _1}z_1,
\dots ,e^{i\theta _d}z_d),
$$
we reduce ourself to the case when
$\alpha _j=-\beta _j$ and $0<\beta _j<\frac \pi 4$.

\par

Let
$$
H(z)=F(z)G_c(z),
$$
where
$$
G_c(z) = \prod _{j=1}^d
e^{-crz_j^{\frac {2\sigma}{\sigma +1}}},
\ s=\flat _\sigma
\quad \text{and}\quad
G_c(z) = \prod _{j=1}^d
e^{-cr(\log (1+z_j))^{\frac {1}{1-2s}}},
\ s\in \mathbf R,\ s<\frac 12.
$$
By choosing $c>0$ large enough and independently
of $r$ we have that $H$ is bounded on $\Omega _0$
and
$$
|H(z)|\lesssim e^{r|z|^2},\quad z\in \Omega .
$$
By Lemma \ref{Lemma:EssEst} it follows that
$H$ is bounded on $\Omega$. This implies
$$
|F(z)|\lesssim |G_c(z)|^{-1},
$$
which gives the desired estimate and thereby
the result.
\end{proof}

\par

\begin{proof}[Proof of Theorem \ref{Thm:Mainthm1}]
Evidently, (3) implies (2), and by Proposition \ref{Prop:PilSpGivesFTEst}
it follows that (1) implies (3). We need to prove that (2) implies (1).

\par

By considering $\mascF _{t_0}f$ and using the fact that $\maclH _s(\rr d)$
and $\maclH _{0,s}(\rr d)$, we reduce ourself to the case that $t_0=0$.

\par

Let $t=(k_1u_1,\dots ,k_du_d)\in \Lambda _{0,u}$ be fixed,
$$
\alpha _j = \textstyle{\frac {\pi k_ju_j}2},
\quad 
\beta _j = \min (\textstyle{\frac {\pi (k_j+1)u_j}2},\pi),
$$
and let $\Gamma _j$, $\Omega$ and $\Omega _0$ be as in
Proposition \ref{Prop:EssEst}, $j=1,\dots ,d$. Also let
$F=(\mathfrak V_df)$. We claim that
\eqref{Eq:EssEst} holds.

\par

In fact, if $\big ( (k_1+1)u_1,\dots , (k_d+1)u_d \big ) \in \Lambda _{0,u}$, then
\eqref{Eq:EssEst} follows from \eqref{Eq:PilFTChar1}
and letting $z_j=x_j$ be real in \eqref{Eq:FTEstGivesWeakPilSpProp2},
$j=1,\dots ,d$.

\par

We only prove the assertion in the Roumieu case.
That is, 
\eqref{Eq:PilFTChar1} holds for some $r>0$, then we shall
prove that $f\in \maclH _s(\rr d)$, or equivalently, that
\begin{equation}\label{Eq:PilSpaceBargEquiv}
\nm {\mathfrak V_df\cdot \omega _{d,r,s}}{L^\infty} <\infty 
\end{equation}
holds for some $r>0$. The other cases follow by similar arguments
and are left for the reader.

\par

Let $k_0\in \{ 1,\dots ,N \}$ be fixed and consider the sector
$\Omega _0$, given by
\begin{align}
\Omega _0 &=\Sets {(z_1,\dots ,z_d)\in \cc d}
{(k_0-1)t_0\frac \pi 2\le \Arg (z_1) \le k_0t_0\frac \pi 2}
\label{Eq:Omega0Def1}
\intertext{or}
\Omega _0 &= \Sets {(z_1,\dots ,z_d)\in \cc d}
{Nt_0\frac \pi 2\le \Arg (z_1) \le \pi}.
\label{Eq:Omega0Def2}
\end{align}
By \eqref{Eq:FTEstGivesWeakPilSpProp2} and the fact that
$(\mascF _2f)(x) = f(-x)$, it follows that
\begin{equation}\label{Eq:FTEstGivesWeakPilSpProp2Cons1}
\sup _{z \in \partial \Omega _0}
|\mathfrak V_df(z)\cdot \omega _{d,r,s}(z)| \le C_r,
\qquad j=1,2.
\end{equation}

\par

Suppose that $k=k_0$ and $z$ belongs to $\Omega _0$ in
\eqref{Eq:Omega0Def1}, or that $k=N+1$ and
$z$ belongs to $\Omega _0$ in \eqref{Eq:Omega0Def2}. Then let
$$
h_{k,c,w_1}(z_1) =
\begin{cases}
\mathfrak V_df(z)\cdot e^{-cr(\log (1+e^{-i(k-\frac 12)t_0\frac \pi 2}z_1))^{\frac 1{1-2s}}},
& s\in \mathbf R,\ 0\le s <\frac 12,
\\[1ex]
\mathfrak V_df(z)\cdot
e^{-cr(e^{-i(k-\frac 12)t_0\frac \pi 2}z_1)^{\frac {2\sigma}{\sigma +1}}},
& s=\flat _\sigma ,
\end{cases}
$$
where $w_1=(z_2,\dots ,z_d)$, and $c>0$ is independent of $r$.
Here we let the branch cut of $\log (\cdo )$
be the negative real axis.

\par

If we choose $c$ large enough and use
\eqref{Eq:FTEstGivesWeakPilSpProp2Cons1},
then it follows that for some constant $c_1$
which only depends on $d$ and $s$, we have
\begin{alignat*}{2}
|h_{k,c,w_1}(z_1)| &\le C_r \omega _{d-1,c_1r,s}(w_1)^{-1}e^{\frac 12|w|^2}, &
\qquad z_1&\in \partial \Omega
\intertext{and}
|h_{k,c,w_1}(z_1)| &\le C_r \omega _{d-1,c_1r,s}(w_1)^{-1}e^{\frac 12|z_1|^2}
e^{\frac 12|w_1|^2}, &
\qquad z_1&\in \Omega .
\end{alignat*}
By Phragm{\'e}n-Lindel{\"o}f's theorem (see Proposition \ref{Prop:PhragmLindel}
in Appendix \ref{AppA}) we have
$$
|h_{k,c,w_1}(z_1)| \le C_r \omega _{d-1,c_1r,s}(w_1)^{-1}e^{\frac 12|w_1|^2},
\qquad z_1\in \Omega ,
$$
which implies that
\begin{equation}\label{Eq:PartEstFirstStep}
|\mathfrak V_df(z) \omega _{1,cr,s}(z_1)\omega _{d-1,cr,s}(w_1)
e^{-\frac 12|w_1|^2}| \le C_r,
\end{equation}
when $z_1\in \Omega _0$ and $w_1\in \cc {d-1}$, for some constant
$c>0$ which is independent of $r$.
Since the union of possible
$\Omega _0$ equals
$$
\Omega = \sets {z_1\in \mathbf C}{\operatorname{Im}(z_1)\ge 0},
$$
the estimate \eqref{Eq:PartEstFirstStep} holds for all such $z_1$. By using
$\mascF _{2+t}f(x)=\mascF _{t}f(-x)$, we obtain similar estimates with
$2+t$ in place of $t$ above, which in turn give \eqref{Eq:PartEstFirstStep}
for $z_1\in \mathbf C\setminus \Omega$. Consequently,
\eqref{Eq:PartEstFirstStep} holds for all $z\in \cc d$.

\medspace

By a combination of \eqref{Eq:FTEstGivesWeakPilSpProp2}
and \eqref{Eq:PartEstFirstStep},
similar arguments with $z_2$ in place of $z_1$ give
\begin{equation}\label{Eq:PartEstSecondStep}
|\mathfrak V_df(z) \omega _{1,cr,s}(z_1)\omega _{1,cr,s}(z_2)
\omega _{d-2,cr,s}(w_2)
e^{-\frac 12|w_2|^2|} \le C_r,
\end{equation}
when $z_1\in \Omega _0$ and $w_2=(z_3,\dots ,z_d)\in \cc {d-2}$,
for some constant $c>0$ which is independent of $r$. By continuing
in this inductively way we obtain after $d$ steps that
$$
|\mathfrak V_df(z) \omega _{1,cr,s}(z_1)\cdots \omega _{1,cr,s}(z_d)|
\le C_r, \qquad z\in \cc d,
$$
for some constant $c>0$ which is independent of $r$. This gives
$$
|\mathfrak V_df(z) \omega _{d,cr,s}(z)|
\le C_r, \qquad  z\in \cc d,
$$
for some constant $c>0$ which is independent of $r$, and the result follows.
\end{proof}

\par

\appendix

\par

\section{Estimates of Phragm{\'e}n-Lindel{\"o}f types}
\label{AppA}

\par

First we recall the following result of Phragm{\'e}n-Lindel{\"o}f. For completeness
we give a proof.

\par

\begin{prop}\label{Prop:PhragmLindel}
Let $\alpha ,\beta ,\rho \in \mathbf R$ be such that
$$
\beta > \alpha ,\quad \rho >0\quad \text{and} \quad
\rho (\beta -\alpha ) <\pi .
$$
Also let $F$ be analytic in
$$
\Omega = \sets {z\in \mathbf C}{\alpha <\Arg (z)<\beta}
$$
and continuous on $\overline{\Omega}$ such that
\begin{alignat*}{2}
|F(z)| &\le M, & \quad \Arg (z) &\in \{ \alpha ,\beta\}
\intertext{and}
|F(z)| &\le Ce^{r|z|^\rho}, & \quad z &\in \Omega ,
\end{alignat*}
for some constants $C,M,r\ge 0$. Then
\begin{equation}\label{Eq:PhragmLindel}
|F(z)| \le M, \quad z\in \overline \Omega .
\end{equation}
\end{prop}

\par

\begin{proof}
By considering $F_0(z)=F(e^{it}z)$ for suitable $t\in \mathbf R$, we reduce
ourself to the case when $\alpha = -\beta <0$. Let $\rho _0>\rho$
be chosen such that
$$
\rho _0\beta = \frac {\rho _0(\beta -\alpha )}2 < \frac \pi 2 ,
$$
and let $\theta = \cos (\rho _0\beta )>0$ and
$$
H_\ep (z) = F(z)e^{-\ep z^{\rho _0}},\qquad \ep >0.
$$
Then $|\Arg (z^{\rho _0})|<\frac \pi 2$ which gives
$$
|H_\ep (z)| = |F(z)| |e^{-\ep z^{\rho _0}}| \le |F(z)| |e^{-\ep \theta |z|^{\rho _0}}|.
$$

\par

This gives
$$
|H_\ep (z)|
\le
Ce^{r|z|^\rho -\ep \theta |z|^{\rho _0}}\to 0
\quad \text{as}\quad
z\in \overline \Omega \ \text{and}\ |z|\to \infty ,
$$
because $\rho _0>\rho$. Hence for some $R_0>0$ we have
\begin{equation}\label{Eq:hepEst1}
|H_\ep (z)| <M
\quad \text{as}\quad
z\in \overline \Omega \ \text{and}\ |z|\ge R_0 .
\end{equation}

\par

If $\Omega _0$ is the set of all $z\in \Omega$ such that $|z|\le R_0$,
then $H_\ep$ is analytic in $\Omega _0$, continuous on
$\overline {\Omega _0}$, and $|H_\ep (z)|< M$ on the boundary
$\partial \Omega _0$ of $\Omega _0$. By the maximum modulus
principle we obtain
\begin{equation}\label{Eq:hepEst2}
|H_\ep (z)| <M
\quad \text{as}\quad
z\in \overline \Omega \ \text{and}\ |z|\le R_0 ,
\end{equation}
and a combination of \eqref{Eq:hepEst1} and \eqref{Eq:hepEst2}
shows that
$$
|H_\ep (z)| <M
\quad \text{as}\quad
z\in \overline \Omega .
$$
By letting $\ep$ tending to $0$ we finally
get \eqref{Eq:PhragmLindel}.
\end{proof}

\par

We have now the following multi-dimensional version of
the previous result.

\par

\begin{prop}\label{Prop:PhragmLindelMultDim}
Let $\alpha _j,\beta _j,\rho _j\in \mathbf R$ be such that
$$
\beta _j> \alpha _j,\quad \rho _j>0\quad \text{and} \quad
\rho _j(\beta _j-\alpha _j) <\pi ,\qquad j=1,\dots ,d.
$$
Also let $F$ be analytic in
$$
\Omega = \sets {z\in \cc d}{\alpha _j<\Arg (z_j)<\beta _j, \ j=1,\dots ,d}
$$
and continuous on $\overline{\Omega}$ such that
\begin{alignat*}{2}
|F(z)| &\le M,  &\quad \Arg (z_j) &\in \{ \alpha _j,\beta _j\},
\intertext{for $j=1,\dots ,d$, and}
|F(z)| &\le C\prod _{j=1}^de^{r_j|z _j|^{\rho _j}},  & \quad z &\in \Omega ,
\end{alignat*}
for some constants $C,M,r_j\ge 0$, $j=1,\dots ,d$. Then
\begin{equation}\label{Eq:PhragmLindelMultDim}
|F(z)| \le M, \quad z\in \overline \Omega .
\end{equation}
\end{prop}


\par

\begin{proof}
We shall argue with induction and apply Phragm{\'e}n-Lindel{\"o}f's
theorem of one dimension in the induction steps.

\par

If we choose
$\theta _1,\dots ,\theta _d\in \mathbf R$
in suitable ways and consider
$$
F(e^{i\theta _1}z_1, \dots ,e^{i\theta _d}z_d),
$$
we reduce ourself to the case when
$\alpha _j=-\beta _j$ and $\beta _j>0$. By considering
$$
F\big (z_1^{2/{\rho _1}}, \cdots ,z_d^{2/{\rho _d}} \big ),
$$
and letting $r=\max _{1\le j\le d}(r_j)$, we then reduce ourself
to the case when 
$$
\rho _1=\cdots =\rho _d=2, \quad r_1=\cdots =r_d=r>0
\quad \text{and}\quad
\beta _j<\frac \pi 4.
$$
(Here we assume that the branch cut for logarithms are
the negative real axis.)

\par

We have
$$
\overline \Omega
=
\Gamma _1\times \cdots \times \Gamma _d,
\quad \text{and let}\quad
\Omega _0
=\partial \Gamma _1\times \cdots
\times \partial \Gamma _d,
$$
where
$$
\Gamma _j = \sets {z\in \mathbf C}
{\alpha _j\le \Arg (z)\le \beta _j},
\quad j=1,\dots ,d.
$$
Then it follows from the assumptions that
$$
|F(z)|\lesssim e^{r|z|^2},\ z\in \Omega,
\quad \text{and}\quad |F(z)|\le M,\
z\in \Omega _0,
$$

\par

For any $k\in \{ 1,\dots ,d\}$, let
$$
\Omega _k = \Gamma _1\times \cdots \times \Gamma _k
\times \partial \Gamma _{k+1}\times \cdots \times
\partial \Gamma _d,
$$
$$
w_k=(z_1,\dots ,z_{k-1},z_{k+1},\dots ,z_d)
\in \cc {d-1}.
$$
and let
$$
G_k(z_k) = F(z)=F(z_1,\dots ,z_d)
$$
when $w_k$ is fixed. We claim that
$$
|F(z)|\le M,\quad z\in \Omega _k,
$$
for every $k\in \{ 1,\dots ,d\}$. This in
turn gives the assertion, by choosing $k=d$.

\par

First suppose that $k=1$ and let
$$
w_1 \in \partial \Gamma _2\times \cdots
\times \partial \Gamma _d \subseteq \cc {d-1}
$$
be fixed.
Then it follows from the assumptions that
$G_1$ is analytic on $\Gamma _1$ and satisfies
$$
|G_1(z_1)|\le
C_{w_1}e^{r_0|z_1|^2},\ z_1\in \Gamma _1
\quad \text{and}\quad |G_1(z_1)|
\le M,\ z_1\in \partial \Gamma _1,
$$
for some constant $C_{w_1}$ which only depends
on $w_1$.
By Phragm{\'e}n-Lindel{\"o}f's theorem (Proposition
\ref{Prop:PhragmLindel}) we get
$|G_1(z_1)|\le M$ when $z_1\in \Gamma _1$,
which is the same as
$$
|F(z)|\le M,\quad z\in \Omega _1,
$$
since $w_1$ is arbitrarily chosen. This gives
the claim for $k=1$.

\par

Suppose that the claim holds true for
$k\in \{1,\dots ,d-1\}$ and prove the claim
for $n=k+1$. Let
$$
w_n\in \Gamma _1\times \cdots \times \Gamma _k
\times \partial \Gamma _{k+2}\times \cdots
\times \partial \Gamma _d \subseteq \cc {d-1}.
$$
be fixed. Then
$$
|G_n(z_n)| \le C_{w_n}e^{r_0|z_n|^2},\
z_n\in \Gamma _n
\quad \text{and}\quad
|G_n(z_n)| \le M,\ z_n\in \partial \Gamma _n.
$$
Here the latter inequality follows from
the induction hypothesis. By Phragm{\'e}n-Lindel{\"o}f's
theorem we get
$|G_n(z_n)|\le M$ when $z_n\in \Gamma _n$,
which is the same as
$$
|F(z)|\le M,\quad z\in \Omega _n,
$$
since $w_n$ is arbitrarily chosen.
This gives the claim and thereby the assertion.
\end{proof}

\par

\section{Links between the Bargmann transform,
fractional Fourier transforms and compositions with
short-time Fourier transform}
\label{AppB}

\par

We recall that if $\mascF _t$ denotes the fractional Fourier transform
of order $t\in \rr d$ and
$f\in \maclH _0'(\rr d)$, then
\begin{equation}\label{Eq:BargmannFracFour}
(\mathfrak V_d(\mascF _tf))(z) = (\mathfrak V_df)(e^{-it_1\frac \pi2}z_1,\dots ,
e^{-it_d\frac \pi2}z_d),
\end{equation}
and the link between the Bargmann transform and 
short-time Fourier transform is
\begin{equation}\label{Eq:BargmannSTFT}
\begin{aligned}
V_\phi f(x,\xi ) &= (2\pi )^{-\frac d2}e^{-\frac 14|z|^2}e^{-\frac i2\scal x\xi}
(\mathfrak V_df)(2^{-\frac 12}\overline z),
\\[1ex]
z &= x+i\xi ,\ x,\xi \in \rr d
\end{aligned}
\end{equation}
(see (1.28) in \cite{To11} and )
A combination of these relations give
\begin{equation}\label{Eq:BargmannSTFT2}
\begin{aligned}
(V_\phi (\mascF _tf))(x,\xi )
&= (2\pi )^{-\frac d2}e^{-\frac 14|z|^2}e^{-\frac i2\scal x\xi}
(\mathfrak V_df)(2^{-\frac 12}\overline {U_{d,t}(z)}),
\\[1ex]
U_{d,t}(z) &=(U_{t_1}(z_1),\dots ,U_{t_d}(z_d))
\\[1ex]
U_{t_j}(z_j)
&=
(\cos (t_j\textstyle{\frac \pi 2})x_j + \sin (t_j\textstyle{\frac \pi 2})\xi _j )
+i(-\sin (t_j\textstyle{\frac \pi 2})x_j +\cos (t_j\textstyle{\frac \pi 2})\xi _j),
\\[1ex]
z &= x+i\xi ,\ x,\xi \in \rr d,
\end{aligned}
\end{equation}
and
\begin{equation}\label{Eq:STFTFracFour}
\begin{aligned}
(V_\phi (\mascF _tf))(x,\xi )
&= 
e^{\frac i4(
\Phi _{t_1}(x_1,\xi _1)
+\cdots +\Phi _{t_d}(x_d,\xi _d))}
V_\phi f(A_{t_1}(x_1,\xi _1),
\dots ,A_{t_d}(x_d,\xi _d)),
\\[1ex]
\Phi _{t_j}(x _j,\xi _j)
&=
\sin (2t_j\textstyle{\frac \pi 2})(\xi _j^2-x_j^2)
+2(\cos (2t_j\textstyle{\frac \pi 2}) -1)x_j\xi _j ,
\\[1ex]
A_{t_j}(x _j,\xi _j)
&=
(\cos (t_j\textstyle{\frac \pi 2})x_j + \sin (t_j\textstyle{\frac \pi 2})\xi _j,
-\sin (t_j\textstyle{\frac \pi 2})x_j +\cos (t_j\textstyle{\frac \pi 2})\xi _j),
\\[1ex]
x,\xi &\in \rr d.
\end{aligned}
\end{equation}

\par

\end{document}